\documentclass[12pt]{article}

\def\@eqnnum{\rm ([section].\theequation)}
\usepackage{amsthm}
\usepackage{amsfonts}
\usepackage{pictexwd,dcpic}
\usepackage{amsmath}
\usepackage{amssymb}
\usepackage{color}
\usepackage{amsbsy}
\numberwithin{equation}{section}
\newcommand*{\ilo}{\text{\begin{huge}$\times$\end{huge}}}

\DeclareMathOperator*{\iloczyn}{\ilo}
\DeclareMathOperator{\und}{and}

\newtheorem{df}{Definition}[section]
\newtheorem{p}[df]{Proposition}

\theoremstyle{definition}

\title{Semi-inner product structures for groupoids}
\author{Piotr Multarzy\'{n}ski\\{\it Faculty of Mathematics and Information Science}\\
 {Warsaw University of Technology} \\
{\it 00-661 Warsaw, Pl. Politechniki 1, Poland}\\
e-mail: multarz@mini.pw.edu.pl}
\date{}
\begin{document}
\maketitle
\begin{abstract}
In this paper there are  considered some scalar valued groupoid bihomomorphism structures, 
being in fact the groupoid counterparts of the inner product notion originally defined for vectors. 
These bihomomorphisms, called here the semi-inner products for groupoids, determine non-negative real valued 
functions which fulfill the axioms assumed for a groupoid norm concept \cite{Buliga}. 
\end{abstract}
\section{Introduction}
A groupoid is a small category in which every morphism (arrow) is invertible, see for example \cite{Brandt, Pysiak, Williams}
and references therein. To be more explicite, let $M$ be a set of objects and $G$ be a set of morphisms, usually called arrows. In such a case we speak of a groupoid $G$ over $M$. Each arrow has an associated source  object (domain) and an associated target object (range). For these associations we introduce two mappings
$d\colon G\rightarrow M$ and $r\colon G\rightarrow M$, the so-called {\em source} and  {\em target}
(or {\em domain} and  {\em range}), respectively. 

For any nonempty subsets $P,Q\subset M$ we shall use the notation
 $G_P=d^{-1}(P)$ ($d$-inverse image of $P$), $G^Q=r^{-1}(Q)$ ($r$-inverse image of $Q$)
and $G_P^Q=G_p\cap G^Q$. 
In particular we have the source-fiber $G_p=d^{-1}(\{p\})$ and the target-fiber $G^q=r^{-1}(\{q\})$.   We shall also use the notation $G_p^q=G_p\cap G^q$. One can notice that $G_P^P$ is a sub-groupoid of $G$ over the base $P\subset M$ and,
in particular, $G_p^p$ is a group, the so-called isotropy group of $G$ at $p\in M$, for any point $p\in M$. 
\section{Affine congruence relation in a groupoid}
In this section we analyze an equivalence relation resembling the idea of parallel shift  of an element. 
\begin{df}\label{relacja lambda}
An equivalence relation $\lambda\subset G\times G$ is called the affine congruence in $G$ if 
\begin{enumerate}
\item (congruence)
\begin{equation}
[g_1\lambda g_2\wedge h_1\lambda h_2\wedge g_1h_1, g_2h_2\in G]\Rightarrow g_1h_1\lambda g_2h_2\; ,
\end{equation}
\item (parallelism) 
\begin{equation}
[g_1\lambda g_2\wedge h_1\lambda h_2\wedge g_1h_2, h_1g_2\in G]\Rightarrow g_1h_2\lambda h_1g_2\; .
\label{rownoleglobok}
\end{equation}
\end{enumerate}
\end{df}
The equivalence class of $g\in G$ with respect to  $\lambda$ is denoted as $[g]$. For any $p\in M$ the notation is used $[g]_p=[g]\cap G_p\;$.
\begin{df}\label{lambda-effective}
An affine congruence $\lambda$ in a groupoid  $G$ is 
\begin{enumerate}
\item  complete if 
 $[g]_p\neq\emptyset$, 
\item simple if $[g]_p$ consists of one element at most, 
\end{enumerate}
for any $g\in G$, $p\in M$.
\end{df}
\begin{df}
An affine congruence $\lambda$ in $G$ is efficient, if it is complete and simple.
 A pair $(G, \lambda )$ is said to be a $\lambda$-affine groupoid, if
$\lambda$ is an efficient affine congruence in  $G$. 
\end{df}
An interesting example of an affine congruence relation can be defined with the help of a groupoid homomorphism into a commutative group. 
\begin{p}Let G - grupoid, $(\mathbb A, +)$ - a commutative group, $\theta\colon G\rightarrow\mathbb A$ - a homomorphism.
Then, the relation $\lambda\subset G\times G$, defined by
\begin{equation}\label{theta-paral}
g\lambda h\Longleftrightarrow \theta (g)=\theta (h),
\end{equation}
is an affine congruence in $G$.
\end{p}
\begin{proof}
Evidently, on the strength of (\ref{theta-paral}), $\lambda$ is a congruence relation in $G$. 
Now we show that the condition (\ref{rownoleglobok}) is fulfilled. 
Assume that
$g_1\lambda g_2$, $h_1\lambda h_2$ and $g_1h_2, h_1g_2\in G$. Hence we get $\theta (g_1)=\theta (g_2)$, $\theta (h_1)=\theta (h_2)$ and $\theta (g_1h_2)=\theta (g_1)+\theta (h_2)=\theta (g_2)+\theta (h_1)=\theta (h_1g_2)$.
\end{proof}
{\bf Example 1. }
 Assume $(M,V,l)$ to be an affine space, i.e. $M$ is a nonempty set (of points), $V$ is a linear space and $l\colon M\times V\rightarrow M$, $l(p,v)\equiv p+v$ is such that the affine space axioms are fulfilled. The Cartesian product $G=M\times V$ is a natural example of a groupoid in which, for any $g=(p,v)$, there is $d(g)=p$ and $r(g)=p+v$. For any $(p,v),(q,w)\in G$ such that $r(p,v)=d(q,w)$ there is $(p,v)+(q,w)=(p,v+w)$. 
Since $(V,+)$ is a commutative (additive) group, we can define the groupoid homomorphism $\theta\colon G\rightarrow V$, $\theta (g)=v$. Now, the affine congruence relation $\lambda$ in $G$ we define as 
$g\lambda h\Longleftrightarrow \theta (g)=\theta (h)$, for any $g,h\in G$. One can easily verify that $\lambda$ is efficient and consequently $(G,\lambda )$ is a $\lambda$-affine groupoid. 
\vspace{2mm}

{\bf Example 2. }
Let $G$, $\theta $ and  $\lambda$ be defined as in Example 1. Assume $P\subset M$ be a  proper (i.e. $P\neq M$) nonempty subset and consider the subgroupoid $G_P^P$ of $G$. Suppose that $p, q\in P$, $p\neq q$, $v\in V$, $q=p+v\in P$ and $s=q+v\not\in P$. Then $g=(p,v)\in G_P^P$ and $h=(q,v)\not\in G_P^P$, i.e. $[g]_q=\emptyset$. Consequently, the relation $\lambda$     is not a complete affine congruence in $G_P^P$ and therefore the groupoid $(G_P^P, \lambda )$ is not a $\lambda$-affine groupoid.   
\vspace{2mm}

{\bf Remark:}  A family of homomorphisms 
$\theta_i\colon G\rightarrow{\mathbb A}_i$,  $({\mathbb A}_i, +)$ - a commutative group,  for $i\in I$, can be represented by a single homomorphism $\theta\colon G\rightarrow\mathbb A$, where ${\mathbb A}=\iloczyn\limits_{i\in I}{\mathbb A}_i$ is the Cartesian product of groups and 
\begin{equation}
(\theta(g))_i=\theta_i(g)
\label{theta composed}
\end{equation}
\begin{df}
Suppose that $\theta\colon G\rightarrow\mathbb A$ is a homomorphism such that, for any $g\in G$, the condition $\theta (g)=0$ implies that there exists $p\in M$ such that $g=e_p$ (the neutral element of the group $G_p^p$). Then $\theta$ is said to be a monomorphism.
\label{mono}
\end{df}
\begin{p}Let $\theta\colon G\rightarrow\mathbb A$ be a monomorphism and $\lambda$ be an affine congruence in $G$ defined by (\ref{theta-paral}). Then  $\lambda$ is simple.
\end{p}
\begin{proof}
Suppose $g_1\lambda g_2$ and $g_1,g_2\in G_p$, for some $p\in M$. This means that $\theta (g_1)=\theta (g_2)$ and $\theta (g_1)-\theta (g_2)=\theta (g_1g_2^{-1})=0$. Since $\theta$ is a monomorphism, there is  $g_1g_2^{-1}=e_p$, for some $p\in M$, i.e. $g_1=g_2$.
\end{proof}
\section{Semi-inner product in a groupoid}
In this section we consider a groupoid scalar valued bihomomorphism resembling an inner or semi-inner product  for vector spaces. There are several proposals known as semi-inner products for vector space. In particular, there is a definition in which a semi-inner product does not need to be additive in the second argument. Then, such a semi-inner product, e.g. for an $L^p$-space, can be consistent with an $L^p$-norm which is not given by any inner product, if $p\neq 2$. In this approach we assume the conjugate symmetry property in order to enable the groupoid norm \cite{Buliga} to be defined  in terms of the bihomomorphism considered.
\begin{df}A function $B:G\times G\rightarrow\mathbb F$ is called the bihomomorphism of a  groupoid $G$ in an additive group $({\mathbb F}, +)$ (or  $\mathbb F$-bihomomorphism) if it is a homomorphism of $G$ in $\mathbb F$ with respect to each argument separately. 
\end{df}

In further, we restrict ourselves to the cases ${\mathbb F}={\mathbb R}$ or ${\mathbb F}={\mathbb C}$. 

\begin{df}Bihomomorphism $B:G\times G\rightarrow\mathbb F$ is said to be a  semi-inner product in $G$ if the following conditions are fulfilled:
\begin{enumerate}
\item (conjugate symmetry) $B(g,h)=\overline{B(h,g)}$, for any $g,h\in G$;
\item (positive definiteness) $B(g,g)> 0$, if $g\neq e_p$, for any $p\in M$ and $g\in G$; 
\item (Cauchy-Schwarz inequality) $|B(g,h)|\leq B(g,g)^{1/2}B(h,h)^{1/2}$, for any $g,h\in G$.
\end{enumerate}
\label{ilocz_skalar}
\end{df}
In the case of $\mathbb{F} = \mathbb{R}$, conjugate symmetry of $B$ reduces to symmetry.
\begin{df}
A pair $(G,B)$ is said to be a semi-unitary groupoid if $G$ is a groupoid and $B$ is a semi-inner product in $G$. 
\end{df} 
\begin{p}Suppose $n\in\mathbb N$ and  $\{\theta_i\}_{i=1}^n$ is a family of groupoid homomorphisms $\theta_i\colon G\rightarrow\mathbb C$ such that, for any $g\in G$, the condition $\forall_{i=1}^n\theta_i (g)=0$ implies that $\exists_{p\in M} g=e_p$. Then the formula
\begin{equation}
B(g,h)=\sum_{i=1}^n\theta_i(g)\overline{\theta_i(h)},
\label{formula B by theta}
\end{equation}
defines a semi-inner product in $G$.
\label{B by theta}
\end{p}
\begin{proof}Conjugate symmetry and positive definteness are evident from definition of $B$. In turn, Cauchy-Schwarz inequality for $B$ is a direct consequence of the Cauchy-Schwarz inequality for the scalar product in ${\mathbb C}^n$.
\end{proof}
\begin{p} Suppose $G$ is a transitive groupoid and let $g\in G$, $p\in M$ be fixed elements. The following implication is true:
$$
\forall_{h\in G_p}B(g, h)=0\Rightarrow \forall_{k\in G}B(g, k)=0.
$$
\end{p}
\begin{proof}Let $h\in G_p^q$  and $k\in G_s$. 
Since $G$ is transitive, there exists $l\in G_q^s$ and $hlk\in G_p$. Consequently, we get $B(g,hlk)=0$ and $B(g,hl)+B(g,k)=0$, which means that $B(g,k)=0$.
\end{proof}
In the case when $G$ is a transitive groupoid, the above statement allows us to weaken the condition 3 of Definition \ref{ilocz_skalar}, i.e. instead $h\in G$ we can take $h\in G_s$, for a fixed point $s\in M$.
\begin{df}
We say that $g_1, g_2\in G$ are $B$-congruent elements if
\begin{equation}
\forall_{h\in G}\; B(g_1,h)=B(g_2,h).
\label{B-congruence relation}
\end{equation} 
In such a case we shall write $g_1\uparrow\uparrow g_2$.
\label{parallel}
\end{df}
Let $\{\theta_i\}_{i=1}^n$ be a family of groupoid homomorphisms which fulfills the assumptions of 
Proposition \ref{B by theta} and suppose $B(g,h)$ is given by formula (\ref{formula B by theta}). Additionally, suppose that for any $j=1,...,n$ there exists $h_j\in G$ such that $\theta_i(h_j)=\delta_{ij}$ (the Kronecker delta). Then the $B$-congruence relation  $\uparrow\uparrow $, defined by formula (\ref{B-congruence relation}), can be equivalently characterized by formula (\ref{theta-paral}). Indeed, for $g_1\uparrow\uparrow g_2$, in the condition $\forall_{h\in G} B(g_1,h)=B(g_2,h)$ we can substitute $h=h_j$ and obtain 
$\theta_j(g_1)=\theta_j(g_2)$, for any $j=1,...,n$.  Hence,  using the groupoid homomorphism $\theta$ given by formula
\ref{theta composed}, we can write equivalently $\theta (g_1)=\theta (g_2)$. On the other hand, if $\theta (g_1)=\theta (g_2)$, i.e. $\forall_{j=1,...,n}\theta_j(g_1)=\theta_j(g_2)$, 
for any $h\in G$
there is 
$B(g_1,h)=\sum_{i=1}^n\theta_i(g_1)\overline{\theta_i(h)}=
\sum_{i=1}^n\theta_i(g_2)\overline{\theta_i(h)}=B(g_2,h)$.
\begin{p} Suppose that $G$ is a transitive groupoid. Then, in Definition \ref{parallel} we can  replace $h\in G$ equivalently by $h\in G_s$, for some $s\in G$.
\end{p}
\begin{proof}Let $s\in M$ and suppose that the identity $B(g_1,h)=B(g_2,h)$ is true for some
elements $g_1,g_2\in G$ and $h\in G_s$. Now, by the tratsivity of $G$, for any $k\in G$, there exists $l$ such that $hl, lk\in G$.
Since $hl, hlk\in G_s$ we can write equalities $B(g_1,hl)=B(g_2,hl)$ and
$B(g_1,hlk)=B(g_2,hlk)$. Then, by the bihomomorphism properties we get the equivalent identity 
$B(g_1,hl)+B(g_1,k)=B(g_2,hl)+B(g_2,k)$ and finally $B(g_1,k)=B(g_2,k)$.
\end{proof}
\begin{p} The relation $\uparrow\uparrow$ is an affine congruence  in $G$.
\end{p}
\begin{proof}First, let us notice that $\uparrow\uparrow$ is an equivalence relation in $G$. Indeed, 
reflexivity and symmetry properties of this relation are evident. To show the transitivity, we can write the conditions 
$g_1\uparrow\uparrow g_2$ and $g_2\uparrow\uparrow g_3$ as $B(g_1,h)=B(g_2,h)$ and $B(g_2,h)=B(g_3,h)$. Hence $B(g_1,h)=B(g_3,h)$, for any $h\in G$, which means that $g_1\uparrow\uparrow g_3$. Suppose now that $g_1\uparrow\uparrow g_2$, $h_1\uparrow\uparrow h_2$ and $g_1h_1, g_2h_2\in G$. Then, for any $k\in G$, we can write equivalently
$B(g_1,k)=B(g_2,k)$,  $B(h_1,k)=B(h_2,k)$ and 
$B(g_1,k)+B(h_1,k)=B(g_2,k)+B(h_2,k)$. Since $B$ is a bihomomorphism, we obtain
$B(g_1h_1,k)=B(g_2h_2,k)$, i.e.   $g_1h_1\uparrow\uparrow g_2h_2$. In turn, if 
$g_1\uparrow\uparrow g_2$, $h_1\uparrow\uparrow h_2$ and $g_1h_2, h_1g_2\in G$, we can write
$B(g_1,k)=B(g_2,k)$,  $B(h_1,k)=B(h_2,k)$ and $B(g_1,k)+B(h_2,k)=B(h_1,k)+B(g_2,k)$ or equivalently
$B(g_1h_2,k)=B(h_1g_2,k)$, i.e. $g_1h_2\uparrow\uparrow h_1g_2$.
\end{proof}
\begin{p}
Suppose $p\in M$ and $g_1,g_2\in G_p$. If  $g_1\uparrow\uparrow g_2$ then $g_1=g_2$, i.e.  the relation $\uparrow\uparrow $ is simple.
\label{g1g2-p-identy}
\end{p}
\begin{proof} On the strength of  $B$-congruence relation between elements $g_1, g_2$, for any $h\in G$ there is 
$B(g_1,h)=B(g_2,h)$ and automatically $-B(g_1,h)+B(g_2,h)=0$, then $B(g_1^{-1},h)+B(g_2,h)=0$ and
$B(g_1^{-1}g_2,h)=0$, which means that  $g_1^{-1}g_2=e_q$ for  $q=r(g_1)\in M$. Finally, we get $g_1=g_2$.
\end{proof}

 In the sequel, we  shall write $[h]=\{g\in G: g\uparrow\uparrow h\}$ and $[h]_p=[h]\cap G_p$, for  $p\in M$. 

In general, the relation $\uparrow\uparrow$ may not be complete, i.e.
there might  be points $p\in M$ where $[h]_p=\emptyset$. For example, let 
$\Gamma ={\mathbb R}^2$ be the groupoid (over the base $\mathbb R$) of pairs, $(x,y)\circ (y,z)=(x,z)$, for any 
$(x,y),(y,z)\in \Gamma$. Define the groupoid homomorphism $\theta\colon \Gamma\rightarrow{\mathbb R}$, by formula $\theta ((x,y))=x-y$, for any $(x,y)\in\Gamma$, and the bihomomorphism  
$B\colon \Gamma\times\Gamma\rightarrow{\mathbb R}$ defined by 
$B((x_1,y_1),(x_2,y_2))=\theta ((x_1,y_1))\theta ((x_2,y_2))$, for any $(x_1,y_1),(x_2,y_2)\in\Gamma$. In fact, the bihomomorphism $B$ is a semi-inner product in $\Gamma$. Choose
$G\subset\Gamma$   to be a square-shaped subgroupoid, e.g.  
$G=\{(x,y)\in{\mathbb R}^2\colon  0\leq x\leq 1, 0\leq y\leq 1\}$ over the base 
$M=\{x\in{\mathbb R}\colon 0\leq x\leq 1\}$. One can easily check that for the element  $(1,0)\in G$ and any $x\in M$, $x\neq 1$, there is $[(1,0)]_x=\emptyset$.
This example motivates the following definition.
\begin{df} $(G,B)$ is said to be a  $B$-affine groupoid if it is a semi-unitary groupoid and the relation $\uparrow\uparrow$ is complete, i.e. for any  $h\in G$ and $p\in M$, there is $[h]_p\neq\emptyset$.
\label{G-afiniczny}
\end{df}
As a consequence of  Proposition \ref{g1g2-p-identy}, in the case of a $B$-affine groupoid $G$, all classes $[h]_p$ are singletons, for any $h\in G$, $p\in M$.
\begin{df}
If $B(g_1,h)=-B(g_2,h)$ for certain $g_1, g_2\in G$ and any $h\in G$, we say that $g_1, g_2$ are $B$-opposite. For $B$-opposite elements $g_1, g_2\in G$ we shall write $g_1\uparrow\downarrow g_2$.
\label{antyparallel}
\end{df}
{\bf Remark:} Evidently, if $g_1\uparrow\downarrow g_2$ and $g_2\uparrow\downarrow g_3$ then $g_1\uparrow\uparrow g_2$, for any $g_1,g_2,g_3\in G$.
\begin{df}
If $g,h\in G$ and $B(g,h)=0$, we say that $g$ and $h$ are $B$-orthogonal and we write $g\bot h$. 
\label{orto}
\end{df}
\begin{df}For any element $g\in G$ and $c\in\mathbb C$, 
we define the subset 
\begin{equation}
G(c,g)\equiv c\bullet g=\{k\in G\colon B(k,h)=c\cdot B(g,h),\; h\in G\}\subset G. 
\end{equation}
Additionally, for $p\in M$ we define $G_p(c,g)=(c\bullet g)_p=G(c,g)\cap G_p$.
\end{df}
If $c=0$, we get $G(0,g)=\{e_p\in G\colon p\in M\}$.
\begin{p}For any $c\in\mathbb C$, $g\in G$ and $k_1,k_2\in G(c,g)$, there is $k_1\uparrow\uparrow k_2$.
\end{p}
\begin{proof}
Let  $h\in G$ be an arbitrary element. Then, by definition we obtain $B(k_1,h)=cB(g,h)=B(k_2,h)$, i.e. $B(k_1,h)=B(k_2,h)$.
\end{proof}
It is possible that the set $G(c,h)\subset G$ is empty, e.g. when $B$ is a real valued bihomomorphism then $G(i,h)=\emptyset$.
\begin{p}
For any $p\in M$, $c\in\mathbb C$ and $g\in G$ the set $G_p(c,g)$ has at most one element.
\end{p}
\begin{proof}
For any two elements $k_1,k_2\in G(c,g)$ there is $k_1\uparrow\uparrow k_2$. Then, on the strength of 
Proposition \ref{g1g2-p-identy}, we have $k_1=k_2$.
\end{proof}
\begin{p}
For any $c\in\mathbb C$, $g,h\in G$ and $k\in G(c,h)$, there is
\begin{equation}
B(g,k)=\overline{c}\cdot B(g,h)\; .
\end{equation}
\end{p}
\begin{proof}
Let $c\in\mathbb C$, $g,h\in G$ and $k\in G(c,h)$. Then we get
$$
B(g,k)=\overline{B(k,g)}=\overline{c\cdot B(h,g)}=\overline{c }\cdot {B(g,h)}\; .
$$
\end{proof}
\section{Norm structure in a grupoid}
The concept of a norm defined in a groupoid is studied by M. Buliga in \cite{Buliga}. 
\begin{df} By a normed groupoid we understand a pair $(G,||\cdot ||)$, where $G$ is a groupoid and the function $||\cdot ||:G\rightarrow [0,+\infty )$, the so-called groupoid norm in $G$, fulfills the conditions:
\begin{enumerate}
\item $||g||=0$ iff there exists $p\in M$ such that $g=e_p\,$,
\item for any $g,h\in G$, such that $gh\in G$,  there is $||gh||\leq ||g||+||h||$,
\item for any $g\in G$ there is $||g^{-1}||=||g||\,$.
\end{enumerate}
\label{normabuliga}
\end{df}
Condition 2 in the above definiotion is called the triangle inequality. One can show that also the reverse triangle inequality holds for $||\cdot ||$, i.e. 
\begin{equation}
|\; ||h||-||g||\; |\leq ||g^{-1}h||\; ,
\label{reversetriangle}
\end{equation} 
for any $g,h\in G$, such that $g^{-1}h\in G$. Indeed, $||h||=||gg^{-1}h||\leq ||g||+||g^{-1}h||$, which gives us $||h||-||g||\leq ||g^{-1}h||$. Analogously, $||g||=||hh^{-1}g||\leq ||h||+||h^{-1}g||$, which implies $||g||-||h||\leq ||h^{-1}g||$. Since $||g^{-1}h||=||h^{-1}g||$, there is $|\; ||h||-||g||\; |\leq ||g^{-1}h||$.

\begin{p} Let $B$ be a semi-inner product in $G$. Then the mapping $||\cdot ||\colon G\rightarrow [0, +\infty )$, defined by 
$||g||=B(g,g)^{1/2}$, for any $g\in G$, is a groupoid norm in $G$.
\end{p}
\begin{proof}
Suppose $||g||=0$, for some $g\in G$. By definition it means that $B(g,g)=0$ and positive definiteness of $B$ implies that $g=e_p$, for some $p\in M$. 

In turn, we can calculate
$||gh||^2=B(gh,gh)=B(g,g)+B(g,h)+B(h,g)+B(h,h)=B(g,g)+B(g,h)+\overline{B(g,h)}+B(h,h)\leq ||g||^2+2|B(g,h)|+||h||^2\leq ||g||^2+2||g||\cdot ||h||+||h||^2=(||g||+||h||)^2$, hence $||gh||\leq ||g||+||h||$.
Finally, $||g^{-1}||=B(g^{-1},g^{-1})=B(g,g)=||g||$.
\end{proof}
\begin{p}Let $|| \cdot ||$ be the groupoid norm defined by a semi-inner product $B$ in a groupoid $G$. Then, for any $c\in\mathbb C$, $g\in G$ and $h\in c\bullet g$, there is
\begin{equation}
||h||=|c|\cdot ||g||\; .
\end{equation}
\end{p}
\begin{proof}If $h\in c\bullet g$, we can write 
$$||h||=B(h,h)^{1/2}=B(cg,cg)^{1/2}=|c|\cdot B(g,g)^{1/2}=|c|\cdot ||g||\; .$$
\end{proof}
\begin{df} Suppose that  $(G,||\cdot ||)$ is a normed groupoid and $\lambda$ is an affine congruence relation in $G$.
We say that $||\cdot||$ is  $\lambda$-consistent if, for any $g_1,g_2\in G$, the following conditions are fulfilled:
\begin{enumerate}
\item $g_1\lambda g_2 \Rightarrow ||g_1||=||g_2||$, 
\item {[$g_1\lambda g_2\wedge g_1g_2\in G$]  $\Rightarrow  ||g_1g_2||=2||g_1||$.}
\end{enumerate}
\end{df}
\begin{p}
Assume $B$ to be a semi-inner product in $G$ and  $\uparrow\uparrow$   be the affine congruence relation in $G$  given by Definition \ref{parallel}. Then  the groupoid norm $||\cdot ||$,  defined by $||g||=B(g,g)^{1/2}$, for $g\in G$, is $\uparrow\uparrow$-consistent. 
\end{p}
\begin{proof}
Let $g_1, g_2\in G$ and suppose that $g_1\uparrow\uparrow g_2$. Then $B(g_1,h)=B(g_2,h)$, for any $h\in G$. In particular, we obtain 
$B(g_1,g_1)=B(g_2,g_1)$ and $B(g_1,g_2)=B(g_2,g_2)$, i.e. $B(g_1,g_2), B(g_2,g_1)\in\mathbb R$, which implies
the identity $||g_1||=||g_2||$. If $g_1\uparrow\uparrow g_2$ and $g_1g_2\in G$, we obtain 
$||g_1g_2||^2=B(g_1g_2,g_1g_2)=B(g_1,g_1)+B(g_1,g_2)+B(g_2,g_1)+B(g_2,g_2)=4||g_1||^2$. Thus we get 
$||g_1g_2||=2||g_1||$.
\end{proof}
\begin{df}\label{parallelogram identiy}Suppose $\lambda$ is an affine congruence relation  and  $||\cdot ||$ is a $\lambda$-consistent groupoid norm in $G$. We say that $||\cdot ||$ fulfills the parallelogram identity with some elements $g,h\in G$, if there exist $g_1,g_2, h_1,h_2\in G$ such that $g_1h_1, g_2^{-1}h_2\in G$,
$g\lambda g_1\lambda g_2$, $h\lambda h_1\lambda h_2$ 
 and
\begin{equation}
||g_1h_1||^2+||g_2^{-1}h_2||^2=2||g||^2+2||h||^2.
\label{parallelogram identiy}
\end{equation}
If the identity (\ref{parallelogram identiy}) is fulffiled with any elements $g,h\in G$, we say that $||\cdot ||$ fulfills the parallelogram identity in $G$.
\end{df}
\begin{p}The groupoid norm $||\cdot ||$ defined by a semi-inner product $B$ in $G$ fulfills the parallelogram identity with all elements $g,h\in G$, for which there exist $g_1, g_2, h_1,h_2\in G$ such that $g_1h_1, g_2^{-1}h_2\in G$
and $g\lambda g_1\lambda g_2$, $h\lambda h_1\lambda h_2$.
\end{p}
\begin{proof} Let $\uparrow\uparrow$ be the affine congruence relation in $G$ given by  Definition \ref{parallel} and assume $g,g_1,g_2, h,h_1,h_2\in G$ be elements such that $g_1h_1, g_2^{-1}h_2\in G$,  $g\uparrow\uparrow g_1\uparrow\uparrow g_2$,  $h\uparrow\uparrow h_1\uparrow\uparrow h_2$.  
Then, by definition we obtain two identities
$$||g_1h_1||^2=B(g_1h_1,g_1h_1)=||g||^2+B(g,h)+B(h,g)+||h||^2\; ,$$ 
$$||g_2^{-1}h_2||^2=B(g_2^{-1}h_2,g_2^{-1}h_2)=||g^{-1}||^2+B(g^{-1},h)+B(h,g^{-1})+||h||^2\; ,$$ 
from which we get
$$||g_1h_1||^2+||g_2^{-1}h_2||^2=$$
$$=||g||^2+B(g^{-1},h)+B(h,g^{-1})+||h||^2+||g||^2+B(g,h)+B(h,g)+||h||^2=$$
$$=2||g||^2+2||h||^2\; .$$
\end{proof}
\begin{p}
Suppose $\lambda$ is an affine congruence relation in G and $||\cdot ||$ is a $\lambda$ -consistent groupoid norm which fulfills the parallelogram identity (\ref{parallelogram identiy}) in $G$. Then the polarization formula 
\begin{equation}
B(g,h)=\frac{1}{4}(||g_1h_1||^2-||g_2^{-1}h_2||^2),
\label{polarization-real}
\end{equation}
for $g,g_1,g_2,h, h_1,h_2\in G$ such that 
$g\lambda g_1\lambda g_2$,
 $h\lambda h_1\lambda h_2$ and
$g_1h_1, g_2^{-1}h_2\in G$, defines the real-valued semi-inner product in $G$.
\end{p}
\begin{proof} 
The proof we start with the observation, that $B(g,h)$ is well defined for any elements $g,h\in G$. Indeed, by the assumption there exist elements
$g_1,g_2, h_1,h_2\in G$ such that 
$g\lambda g_1\lambda g_2$,
 $h\lambda h_1\lambda h_2$ and
$g_1h_1, g_2^{-1}h_2\in G$, which is enough to define $B(g,h)$ by formula (\ref{polarization-real}). This definition is correct, since $||\cdot ||$ is $\lambda$-consistent.

Let $g\lambda g_i$ and $h\lambda h_i$, for some $g,g_i,h,h_i\in G$, $i=1,2,3,4$,  such that $g_1h_1,g_2^{-1}h_2, h_3g_3, h_4^{-1}g_4\in G$. Since $\lambda$ is an affine congruence relation in $G$, we have $g_1h_1\lambda h_3g_3$ and $g_2^{-1}h_2\lambda h_4^{-1}g_4$. On the other hand, 
since $||\cdot ||$ is $\lambda$-consistent, there is $||g_1h_1||=||h_3g_3||$ and
 $||g_2^{-1}h_2||=||h_4^{-1}g_4||$. 
Hence, we obtain the (conjugate) symmetry property  
$$B(g,h)=\frac{1}{4}(||g_1h_1||^2-||g_2^{-1}h_2||^2)=
\frac{1}{4}(||h_3g_3||^2-||h_4^{-1}g_4||^2)=B(h,g).$$ 

Let $k_1, k_2\in G$ be elements such that $g\lambda k_1\lambda k_2$ and $k_1k_2\in G$. Then,  
the positive definitness property of $B$ is clear from $$B(g,g)=\frac{1}{4}(||k_1k_2||^2-||g^{-1}g||^2)=\frac{1}{4}(2||g||)^2=||g||^2.$$

To prove the Cauchy-Schwarz inequality for $B$, we shall use the parallelogram identity (\ref{parallelogram identiy}) assumed 
for $||\cdot ||$.
From the reverse triangle inequality (\ref{reversetriangle}) we obtain 
$$(\; ||g||-||h||\; )^2=||g||^2-2||g||\cdot ||h||+||h||^2\leq ||g_2^{-1}h_2||^2\; , $$
$$2||g||^2+2||h||^2-||g_2^{-1}h_2||^2-||g_2^{-1}h_2||^2\leq 4||g||\cdot ||h||\; .$$
Now, on the strength of the parallelogram identity (\ref{parallelogram identiy}), from the last inequality we obtain
$$||g_1h_1||^2-||g_2^{-1}h_2||^2\leq 4||g||\cdot ||h||\; ,$$
which means that
$$
B(g,h)=\frac{1}{4}(||g_1h_1||^2-||g_2^{-1}h_2||^2)\leq ||g||\cdot ||h||\; .
$$

In order to complete this proof, one needs to show that $B(g,h)$ defined by the polarization identity is a bihomomorphism. Since the (conjugate) symmetry of $B(g,h)$ has been already proved, it is enough to show that $B(gh,k)=B(g,k)+B(h,k)$.

Directly from the polarization identity (\ref{polarization-real}) one can notice that
$$B(g,e_p)=\frac{1}{4}(||ge_{r(g)}||^2-||g^{-1}e_{d(g)}||^2)=
\frac{1}{4}(||g||^2-||g^{-1}||^2)=0,$$ for any $g\in G$ and $p\in M$.

Choose elements $u,v,x,y,u_i,v_i,x_i,y_i\in G$,  such that $u\lambda u_i$, $v\lambda v_i$,$x\lambda x_i$,$y\lambda y_i$, for 
$i=1,2,...$, and assume the composability of elements in all cases that appear in formulae below, in particular 
$u_1v_1, u_2^{-1}v_2, x_1y_1, x_2^{-1}y_2\in G$. 

Then,  using the polarization  identity (\ref{polarization-real}), we can write 
$$B(u,v)=\frac{1}{4}(||u_1v_1||^2-||u_2^{-1}v_2||^2)\; \und\; B(x,y)=\frac{1}{4}(||x_1y_1||^2-||x_2^{-1}y_2||^2).$$ 
Now, we apply the parallelogram identity (\ref{parallelogram identiy}) and calculate the sum
$$B(u,v)+B(x,y)=\frac{1}{4}(||u_1v_1||^2+||x_1y_1||^2)-\frac{1}{4}(||u_2^{-1}v_2||^2+||x_2^{-1}y_2||^2)=$$ 

$$=\frac{1}{4}\cdot\frac{1}{2}(||(u_3v_3)(x_3y_3)||^2+||(u_4v_4)^{-1}(x_4y_4)||^2)-$$
$$-\frac{1}{4}\cdot\frac{1}{2}(||(u_5^{-1}v_5)(x_5^{-1}y_5)||^2+||(u_6^{-1}v_6)^{-1}(x_6^{-1}y_6)||^2)=$$

$$=\frac{1}{4}\cdot\frac{1}{2}(||(u_3v_3)(x_3y_3)||^2-||(u_5^{-1}v_5)(x_5^{-1}y_5)||^2)+$$
$$+\frac{1}{4}\cdot\frac{1}{2}(||(u_4v_4)^{-1}(x_4y_4)||^2-||(u_6^{-1}v_6)^{-1}(x_6^{-1}y_6)||^2)=$$

$$=\frac{1}{2}\cdot\frac{1}{4}(||(u_7x_7)(v_7y_7)||^2-||(u_8x_8)^{-1}(v_8y_8)||^2)+$$
$$+\frac{1}{2}\cdot\frac{1}{4}(||(u_9^{-1}x_9)(v_9^{-1}y_9)||^2-||(u_{10}^{-1}x_{10})^{-1}(v_{10}^{-1}y_{10})||^2)=$$

$$=\frac{1}{2}B(u_8x_8,v_8y_8)+\frac{1}{2}B(u_9^{-1}x_9,v_9^{-1}y_9).$$ 

Since $B(u,e_p)=0$, from the above result we obtain
$B(u,v)=B(u,v)+B(u,e_p)=
\frac{1}{2}B(u_1u_2,ve_{r(v)})+\frac{1}{2}B(u^{-1}u,v^{-1}e_{d(v)})=\frac{1}{2}B(u_1u_2,v)$.

Let us make the substitutions  $u=g, x=h, v=y=k$ and assume that $gh\in G$. Then we get
$$B(g,k)+B(h,k)=\frac{1}{2}B(gh,k_1k_2)+\frac{1}{2}B(g^{-1}h_1,k^{-1}k)=B(gh,k).$$ 
\end{proof}



\begin{thebibliography}{11}
\bibitem{Brandt}H. Brandt, {\em \"Uber eine Verallgemeinerung der Gruppen-Begriffes}, Math. Ann., 96, 360-366 (1926).
\bibitem{Buliga} M. Buliga, {\em Normed groupoids with dilations},
arXiv:1107.2823v1 [math.MG] 14 Jul 2011.
\bibitem{Ivan}M. Ivan, {\em General properties of the symmetric groupoid of a finite set}, Annals of University of Ceaiova, Math. Comp. Sci. Ser. Vol. 30(2), 2003, 109-119. 
\bibitem{Mackenzie} K.C.H. Mackenzie, {\em Affinoid structures and connections}, Poisson Geometry, Banach Center Publications, vol. 51, Warsaw 2000.
\bibitem{groupoid}A. Weinstein, {\em Groupoids: unifying internal and external symmetry, a tour through some examples}, Notices of the Amer. Math. Soc. 43 (1996), 744-752. 
\bibitem{Pysiak}L. Pysiak, {\em Groupoids, their representations and imprimitivity systems}, 
Demonstratio Math. 37 (2004), 661-670.
\bibitem{Williams} M.B. Williams, {\em Introduction to groupoids}, \\http://www.ma.utexas.edu/users/mwilliams/groupoids.pdf
\end{thebibliography}
\end{document}